\documentclass[a4paper, reqno, 11pt]{amsart}

\makeatletter

\@addtoreset{equation}{section}
\makeatother
\usepackage{setspace}
\usepackage{xcolor}
\usepackage{amsmath}
\usepackage{amssymb}
\usepackage{latexsym}
\usepackage{amsthm}
\usepackage{hyperref}
\newtheorem{Thm}{Theorem}[section]

\newtheorem{Lem}[Thm]{Lemma}
\newtheorem{Prop}[Thm]{Proposition}

\theoremstyle{definition}

\newtheorem{Rem}[Thm]{Remark}

\begin{document}

\title[]{An equivalence criterion for infinite products of Cauchy measures}
\author{Kazuki Okamura}
\date{}
\address{School of General Education, Shinshu University, 3-1-1, Asahi, Matsumoto, Nagano, 390-8621, JAPAN.} 
\email{kazukio@shinshu-u.ac.jp} 
\maketitle

\begin{abstract}

We give an equivalence-singularity criterion for infinite products of Cauchy measures under simultaneous shifts of the location and scale parameters.
Our result is an extension of  Lie and Sullivan's result  giving an equivalence-singularity criterion  under dilations of scale parameters.  
Our proof utilizes McCullagh's parameterization of the Cauchy distributions and maximal invariant, and a closed-form formula of the Kullback-Leibler divergence between two Cauchy measures given by Chyzak and Nielsen. 

\end{abstract}

\section{Introduction}

The purpose of this note is to establish an equivalence criterion for two infinite products of Cauchy measures. 
Let $(Z_n)_n$ be independent Cauchy random variables with location $z_{n1} \in \mathbb{R}$ and scale $z_{n2} > 0$, and   $(W_n)_n$ be independent Cauchy random variables with location $w_{n1} \in \mathbb{R}$ and scale $w_{n2} > 0$.  
We will give a necessary and sufficient condition for the equivalence of the laws of $(Z_n)_n$ and $(W_n)_n$ in terms of $(z_{n1}, z_{n2}, w_{n1}, w_{n2})_n$. 

The Cameron-Martin theorem characterizes the directions in which an infinite-dimensional Gaussian measure can be translated under the constraint that the translations  preserve equivalence of the original and translated measures. 
Hajek \cite{Hajek1958a, Hajek1958b} and Feldman \cite{Feldman1958} establish a dichotomy between two Gaussian measures, which are equivalent or singular to each other. 
Bogachev \cite[Theorems 2.7.2 and 2.12.9]{Bogachev1998} and Shiryaev \cite[Theorem 7.6.5]{Shiryaev2018} also deal with this issue.  
It would be interesting to consider this kind of equivalence-singularity dichotomy  for infinite-dimensional measures other than the infinite-dimensional Gaussian measures.  
Recently, Lie and Sullivan \cite{Lie2018} gave a characterization of equivalence of measures  for infinite-dimensional Cauchy measures which could be regarded as an analogue of the Cameron-Martin theorem.
However, \cite{Lie2018} deals with  dilations of the scale parameters only, more specifically, \cite{Lie2018} considers the case that all locations are identical with each other,  that is, $z_{n1} = w_{n1}$ for every $n \ge 1$. 

We extend  \cite{Lie2018} to a more general case.  
We deal with simultaneous changes of the location and scale parameters. 
In a manner similar to \cite{Lie2018}, 
our proof depends on Kakutani's theorem \cite{Kakutani1948} for the equivalence-singularity dichotomy of infinite product measures. 
However, our proof is different from \cite{Lie2018}. 
Our proof is simpler and shorter. 
Indeed, we do not use a Taylor expansion of an integral for the scale parameter.  
Our main tool is McCullagh's parametrization and maximal invariant \cite{McCullagh1993}. 
We will also use an explicit formula of Kullback-Leibler divergence between two Cauchy distributions obtained by Chyzak and Nielsen \cite{Chyzak2019}. 

\subsection{Framework and main result}

Throughout this note, for $z \in \mathbb{C}$, $\textup{Re}(z)$ and $\textup{Im}(z)$ are the real and imaginary parts of $z$. 
Let $i$ be the imaginary unit. 
Let the upper-half plane $\mathbb{H} := \left\{x+ y i | x \in \mathbb{R}, y > 0 \right\}$. 
For $z \in \mathbb{H}$, 
let $P_z$ be the Cauchy distribution with parameter $z$ equipped with McCullagh's parametrization \cite{McCullagh1993}, that is, 
the density function of $P_z$ is given by $\displaystyle \frac{\textup{Im}(z)}{\pi |x - z|^2}$.
For an infinite sequence $z = (z_n)_n$ in $\mathbb{H}$, 
we let $\displaystyle \bigotimes_{n} P_{z_n}$ be the product probability measure of a sequence of probability measures $(P_{z_n})_n$ on the product space $\mathbb{R}^{\mathbb{N}}$ equipped with the cylindrical $\sigma$-algebra.  
For two measures $\mu$ and $\nu$ on a common probability space such that $\mu$ is absolutely continuous with respect to $\nu$, 
we denote the Radon-Nikodym derivative of $\mu$ with respect to $\nu$ by $\dfrac{d\mu}{d\nu}$. 
We let $E^{P}$ be the expectation with respect to a probability measure $P$. 

\begin{Thm}\label{main}
For every two infinite sequences $z = (z_n)_n$ and $w = (w_n)_n$ of $\mathbb{H}$, we have that\\ 
(i) If $\displaystyle \sum_{n = 1}^{\infty} \frac{|z_n - w_n|^2}{\textup{Im}(z_n) \textup{Im}(w_n)  } < +\infty$, then, $\displaystyle \bigotimes_{n} P_{z_n}$ and $\displaystyle \bigotimes_{n} P_{w_n}$ are equivalent, that is, 
$\displaystyle \bigotimes_{n} P_{z_n}$ is absolutely continuous with respect to $\displaystyle \bigotimes_{n} P_{w_n}$ and vice versa.\\
(ii)   If $\displaystyle \sum_{n = 1}^{\infty} \frac{|z_n - w_n|^2}{\textup{Im}(z_n) \textup{Im}(w_n)  } = +\infty$, then, $\displaystyle \bigotimes_{n} P_{z_n}$ and $\displaystyle \bigotimes_{n} P_{w_n}$ are singular to each other. 
\end{Thm}


\section{Proof}

By adopting a version of Kakutani's theorem as in \cite[Corollary 7.6.3]{Shiryaev2018}  for our framework, we have that 
\begin{Prop}\label{Kakutani}
For every two infinite sequences $z = (z_n)_n$ and $w = (w_n)_n$ of $\mathbb{H}$, we have that\\ 
(i) If $\displaystyle \sum_{n = 1}^{\infty} -\log E^{P_{z_n}}\left[ \sqrt{\frac{d P_{w_n}}{d P_{z_n}}}\right] < +\infty$, then, $\displaystyle \bigotimes_{n} P_{z_n}$ and $\displaystyle \bigotimes_{n} P_{w_n}$ are equivalent, that is, 
$\displaystyle \bigotimes_{n} P_{z_n}$ is absolutely continuous with respect to $\displaystyle \bigotimes_{n} P_{w_n}$ and vice versa.\\
(ii)   If $\displaystyle \sum_{n = 1}^{\infty} -\log E^{P_{z_n}}\left[ \sqrt{\frac{d P_{w_n}}{d P_{z_n}}}\right] = +\infty$, then, $\displaystyle \bigotimes_{n} P_{z_n}$ and $\displaystyle \bigotimes_{n} P_{w_n}$ are singular to each other. 
\end{Prop}


Let 
\begin{equation}\label{chi-def} 
\chi(z,w) := \frac{|z-w|^2}{\textup{Im}(z) \textup{Im}(w)}. 
\end{equation}

Let the Kullback-Leibler divergence be 
\[ K(P_{z}| P_{w}) := E^{P_z}\left[  \log \frac{d P_z}{d P_w} \right], \ \  z, w \in \mathbb H. \]
We remark that this is symmetric, specifically, 
\begin{equation}\label{KL-sym}
K(P_{z}| P_{w}) = K(P_{w}| P_{z}), \ z, w \in \mathbb{H}. 
\end{equation}

\begin{Prop}[Chyzak-Nielsen   {\cite{Chyzak2019}}]\label{CN}
For $z, w \in \mathbb H$, 
\begin{equation}\label{KL}
K\left(P_{z}| P_{w}\right) = \log\left(1 + \frac{\chi(z, w)}{4}\right). 
\end{equation} 
\end{Prop}

Now we show Theorem \ref{main} (i). 
By Jensen's inequality, 
\[ \sum_{n = 1}^{\infty} -\log E^{P_{z_n}}\left[ \sqrt{\frac{d P_{w_n}}{d P_{z_n}}}\right]  \le \frac{1}{2} \sum_{n = 1}^{\infty}  E^{P_{z_n}}\left[  -\log\frac{d P_{w_n}}{d P_{z_n}} \right] = \frac{1}{2} \sum_{n = 1}^{\infty} K\left(P_{z_n}|P_{w_n} \right) \] 
\[ \le \frac{1}{8} \sum_{n = 1}^{\infty} \chi(z_n, w_n) < +\infty. \] 
The assertion follows from this and Proposition \ref{Kakutani} (i).

By \cite{McCullagh1993},  
the function $\chi$ defined in \eqref{chi-def} is a {\it maximal invariant} for the action of the special linear group $SL(2, \mathbb{R})$ to $\mathbb{H} \times \mathbb{H}$ defined by 
$$A \cdot (z,w) := \left(\frac{az+b}{cz+d},  \frac{aw+b}{cw+d}\right), \  \ \ A = \begin{pmatrix} a & b \\ c & d \end{pmatrix} \in SL(2, \mathbb{R}), \ z, w \in \mathbb{H}.$$ 
That is, 
\[ \chi(A \cdot z, A \cdot w) = \chi(z,w), \ \ A \in SL(2, \mathbb{R}), \ z, w  \in \mathbb{H}.\]
and it holds that for every $z, w, z^{\prime}, w^{\prime} \in \mathbb{H}$ satisfying that $\chi(z^{\prime}, w^{\prime}) = \chi(z,w)$, there exists $A \in SL(2, \mathbb{R})$ such that $(A \cdot z, A \cdot w) = (z^{\prime}, w^{\prime})$. 
See Eaton \cite[Chapter 2]{Eaton1989} for more details about maximal invariants.

\begin{Lem}\label{i-inv}
Let 
\[ I(z,w) := \frac{\sqrt{\textup{Im}(z) \textup{Im} (w)}}{\pi} \int_{\mathbb R} \frac{1}{|x - z| |x - w|} dx.  \]
Then, 
\[ I\left(A \cdot (z,  w) \right) = I(z,w), \ \ A \in SL(2, \mathbb{R}), \ z, w  \in \mathbb{H}. \]
\end{Lem}

\begin{proof}
In this proof we let 
$$A \cdot z := \frac{az+b}{cz+d},  \  \ \ A = \begin{pmatrix} a & b \\ c & d \end{pmatrix} \in SL(2, \mathbb{R}), \ z  \in \mathbb{H}.$$ 

Since $|x- z| = \left|(x+\epsilon i)-  (z+\epsilon i)\right|$ and $\chi$ is invariant, we have that for every $\epsilon > 0, x \in \mathbb{R}, z \in \mathbb{H}$, 
\[  \frac{\sqrt{\textup{Im}(z)}}{|x - z| } = \frac{1}{\sqrt{\epsilon \chi(x+\epsilon i,  z+\epsilon i)}}  =  \frac{1}{\sqrt{\epsilon \chi(A \cdot (x+\epsilon i),  A \cdot (z+\epsilon i))}}  \]
If we let $\displaystyle  A = \begin{pmatrix} a & b \\ c & d \end{pmatrix}$, then, 
\[ \textup{Im}\left(A \cdot (x+\epsilon i)\right) = \frac{\epsilon}{(cx+d)^2 + c^2 \epsilon^2}.  \]
Hence, 
\[ \frac{1}{\sqrt{\epsilon \chi(A \cdot (x+\epsilon i),  A \cdot (z+\epsilon i))}}  = \frac{\sqrt{  \textup{Im}\left(A \cdot (z+\epsilon i)\right)      }}{ |A \cdot (x+\epsilon i) - A \cdot (z+\epsilon i)| \sqrt{(cx+d)^2 + c^2 \epsilon^2}}. \]

Therefore, by letting $\epsilon \to +0$, we have that for every $x$ such that $cx +d \ne 0$, 
\[  \frac{\sqrt{\textup{Im}(z)}}{|x - z| } = \frac{\sqrt{\textup{Im}(A \cdot z)}}{|A \cdot x - A \cdot z| |cx+d|}. \]

Therefore by the change-of-variable formula, 
\[ I(z,w)  = \frac{1}{\pi} \int_{\mathbb R} \frac{\sqrt{\textup{Im}(A \cdot z) \textup{Im}(A \cdot w)}}{|A \cdot x - A \cdot z| |A \cdot x - A \cdot w|  (cx+d)^2}  dx = I(A \cdot z, A \cdot w). \]
\end{proof}

By Lemma \ref{i-inv} and \cite[Theorem 2.3]{Eaton1989},  
we see that there exists a unique function $J : [0, \infty)  \to [0, \infty)$ such that 
\begin{equation}\label{rep}
J(\chi(z,w)) = I(z,w), \ z, w \in \mathbb{H}.
\end{equation}  

\begin{Lem}\label{J-diverge}
$\displaystyle \lim_{t \to +\infty} J(t) = 0$.
\end{Lem}

\begin{proof}
By \cite{McCullagh1993},  
we have that for every $t \ge 0$, there exists $\lambda \ge 1$ such that $\chi(\lambda i, i) = t$. 
As a function of $\lambda$, $\chi(\lambda i, i) = \dfrac{(\lambda -1)^2}{\lambda}$ is strictly increasing if $\lambda \ge 1$. 
It holds that for $\lambda \ge 1$, 
\[ \dfrac{\sqrt{\lambda}}{\sqrt{x^2 + \lambda^2}} \frac{1}{\sqrt{x^2+1}} \le \min\left\{1, |x|^{-1/2} \right\} \frac{1}{\sqrt{x^2+1}} \]
and $\displaystyle \min\left\{1, |x|^{-1/2} \right\} \frac{1}{\sqrt{x^2 +1}}$ is integrable on $\mathbb R$. 

Since $\displaystyle \lim_{\lambda \to \infty} \sqrt{\dfrac{\lambda}{x^2 + \lambda^2}} = 0$, 
we have that by the dominated convergence theorem, 
\[ \lim_{\lambda \to +\infty} I(\lambda i, i) = \lim_{\lambda \to +\infty} \frac{1}{\pi} \int_{\mathbb R}  \dfrac{\sqrt{\lambda}}{\sqrt{x^2 + \lambda^2}} \frac{1}{\sqrt{x^2 +1}}  dx = 0.  \]
\end{proof}

\begin{Lem}\label{chi-bdd}
If $\displaystyle \bigotimes_{n} P_{z_n}$ and $\displaystyle \bigotimes_{n} P_{w_n}$ are equivalent, then, $\left\{\chi(z_n, w_n) \right\}_n$ is bounded in $[0, \infty)$. \\
\end{Lem}

\begin{proof}
If $\displaystyle \bigotimes_{n} P_{z_n}$ and $\displaystyle \bigotimes_{n} P_{w_n}$ are equivalent, then, by Proposition \ref{Kakutani}, 
we have that 
\[ \displaystyle \sum_{n = 1}^{\infty} -\log E^{P_{z_n}}\left[ \sqrt{\frac{d P_{w_n}}{d P_{z_n}}}\right] < +\infty, \]
and hence,  
\[ \lim_{n \to \infty} J(\chi(z_n, w_n))=  \lim_{n \to \infty} I(z_n, w_n) = \lim_{n \to \infty} E^{P_{z_n}}\left[ \sqrt{\frac{d P_{w_n}}{d P_{z_n}}}\right] = 1. \]
Now the assertion follows from this and Lemma \ref{J-diverge}. 
\end{proof}

\begin{Lem}\label{log-bdd}
If $\chi(z,w) \le C_1$, then, there exists a constant $C_2$ depending only on $C_1$ such that 
\[ \frac{1}{C_2} \le \inf_{x \in \mathbb{R}} \frac{\textup{Im}(w)}{\textup{Im}(z)} \frac{|x - z|^2}{|x - w|^2} \le \sup_{x \in \mathbb{R}} \frac{\textup{Im}(w)}{\textup{Im}(z)} \frac{|x - z|^2}{|x - w|^2}  \le C_2.  \]
\end{Lem}

\begin{proof}
By the definition of $\chi$, we have that there exist $C_3 \in (1, \infty)$ depending only on $C_1$ such that for every $(z,w)$ satisfying $\chi(z,w) \le C_1$, 
\[ \frac{1}{C_3} \le \frac{\textup{Im}(z)}{\textup{Im}(w)} \le C_3. \]

We have that 
\[ \sup_{x \in \mathbb{R}} \frac{\textup{Im}(w)}{\textup{Im}(z)} \frac{|x - z|^2}{|x - w|^2}  \le \frac{\textup{Im}(z)}{\textup{Im}(w)} + \frac{\textup{Im}(w)}{\textup{Im}(z)} \sup_{x \in \mathbb{R}}  \frac{(x - \textup{Re}(z-w))^2}{x^2 + \textup{Im}(w)^2}  \]
\[ \le \frac{\textup{Im}(z)}{\textup{Im}(w)} + \frac{\textup{Im}(w)}{\textup{Im}(z)} \left( 1 + \frac{\textup{Re}(z-w)^2}{\textup{Im}(w)^2} \right) \le 2 C_3 + C_1. \]

The lower bound follows from the upper bound and symmetry about $z$ and $w$. 
\end{proof}

\begin{proof}[Proof of Theorem \ref{main} (ii)]
By Lemma \ref{chi-bdd} and Proposition \ref{Kakutani}, we have that if $\left\{\chi(z_n, w_n) \right\}_n$ is not bounded, then, $\displaystyle \bigotimes_{n} P_{z_n}$ and $\displaystyle \bigotimes_{n} P_{w_n}$ are singular to each other. 
We assume that $\left\{\chi(z_n, w_n) \right\}_n$ is bounded.

Let $X_i (x) := x_i, \ i \in \mathbb{N}$, for $x = (x_i)_i \in \mathbb{R}^{\mathbb{N}}$.  
Then, $\{X_i\}_i$ are independent under $\bigotimes_n P_{w_n}$. 
By the assumption of Theorem \ref{main} (ii) and \eqref{KL-sym}, 
$$ \sum_{n = 1}^{\infty} E^{\bigotimes_n P_{w_n}} \left[ \log \dfrac{d P_{w_n}}{d P_{z_n}} (X_n) \right] = \sum_{n = 1}^{\infty} E^{P_{w_n}} \left[ \log \dfrac{d P_{w_n}}{d P_{z_n}} \right] = \frac{1}{2} \sum_{n = 1}^{\infty} K(P_{w_n} | P_{z_n})  =  +\infty. $$
By Lemma \ref{log-bdd} and 
\[ \frac{d P_{w}}{d P_{z}} (x) = \frac{\textup{Im}(w)}{\textup{Im}(z)} \frac{|x - z|^2}{|x - w|^2}, \ x \in \mathbb{R},   \]
we have that 
$\left\{ \left| \log \dfrac{d P_{w_n}}{d P_{z_n}} (X_n) \right| \right\}_n$ is uniformly bounded.  
Therefore, by the Kolmogorov three series theorem and the Kolmogorov 0-1 law, 
$\displaystyle \sum_{n = 1}^{\infty} \log \dfrac{d P_{w_n}}{d P_{z_n}} (X_n) $ diverges with probability one under $\displaystyle \bigotimes_{n} P_{w_n}$.  

By \cite[Theorem 7.6.1]{Shiryaev2018}, it holds that with probability one under $\displaystyle \bigotimes_{n} P_{w_n}$, 
$\displaystyle \lim_{N \to \infty} \prod_{n=1}^{N} \dfrac{d P_{w_n}}{d P_{z_n}} (X_n)$ exists by allowing the limit takes $+\infty$.  
By this and the Kolmogorov 0-1 law, 
it holds that $\displaystyle \bigotimes_{n} P_{w_n} \left( \sum_{n = 1}^{\infty} \log \dfrac{d P_{w_n}}{d P_{z_n}} (X_n) = +\infty \right) = 1$ or  
 $\displaystyle \bigotimes_{n} P_{w_n} \left( \sum_{n = 1}^{\infty} \log \dfrac{d P_{w_n}}{d P_{z_n}} (X_n) = -\infty\right) = 1$. 
We have that 
$$ \lim_{N \to \infty} \prod_{n=1}^{N} \dfrac{d P_{w_n}}{d P_{z_n}} (X_n) = \dfrac{d \left(\bigotimes_n P_{w_n} \right)}{d \left(\bigotimes_n P_{z_n} \right)}, \  \ \textup{$ \bigotimes_{n} P_{w_n}$-a.s.}$$ 
and 
$$ \bigotimes_{n} P_{w_n} \left( \dfrac{d \left(\bigotimes_n P_{w_n} \right)}{d \left(\bigotimes_n P_{z_n} \right)} = 0 \right)  = 0.$$
Hence, 
\[  \bigotimes_{n} P_{w_n} \left( \dfrac{d \left(\bigotimes_n P_{w_n} \right)}{d \left(\bigotimes_n P_{z_n} \right)} = +\infty \right)  =  \bigotimes_{n} P_{w_n} \left( \sum_{n = 1}^{\infty} \log \dfrac{d P_{w_n}}{d P_{z_n}} (X_n) = +\infty \right) = 1. \]
By \cite[Theorem 7.6.2]{Shiryaev2018}, 
we have the assertion. 
\end{proof}

\begin{Rem}
We can also show Theorem \ref{main} by using \cite{Lie2018}. 
Let $\lambda_n = \lambda_n (z_n, w_n)$ be the number such that $\chi(z_n, w_n) = \chi(\lambda_n i, i)$.  
Then, by \eqref{rep}, 
$I(z_n, w_n) = I(\lambda_n i, i)$. 
By \cite[Theorems 1.1 and 2.1]{Lie2018},  
\[ \sum_{n = 1}^{\infty} -\log I(\lambda_n i, i) < +\infty  \Longleftrightarrow \sum_{n = 1}^{\infty} (\lambda_n - 1)^2 < +\infty \]
\[ \Longleftrightarrow \sum_{n = 1}^{\infty} \frac{|z_n - w_n|^2}{\textup{Im}(z_n) \textup{Im}(w_n)} = \sum_{n = 1}^{\infty} \frac{(\lambda_n - 1)^2}{\lambda_n} < +\infty. \]
\end{Rem}

{\it Acknowledgement.} \ The author is supported by JSPS Grant-in-Aid 19K14549.

\bibliographystyle{amsplain}
\bibliography{Cauchy-dichotomy}

\end{document}